\documentclass[a4paper,11pt]{article}

\usepackage{amsmath,amsthm,amssymb,mathrsfs,latexsym,amsfonts}
\usepackage{graphicx,psfrag,epsfig}
\usepackage{bbm}
\usepackage{a4wide}
\usepackage[english]{babel}
\usepackage[latin1]{inputenc}
\usepackage{authblk}

\numberwithin{equation}{section}

\newtheorem{theorem}{Theorem}[section]
\newtheorem{lemma}[theorem]{Lemma}
\newtheorem{proposition}[theorem]{Proposition}

\newcounter{paraga}[section]

\newcommand{\N}{\mathbb{N}}
\newcommand{\Z}{\mathbb{Z}}
\newcommand{\Q}{\mathbb{Q}}
\newcommand{\R}{\mathbb{R}}
\newcommand{\C}{\mathbb{C}}

\begin{document}

\def\MP{\,{<\hspace{-.5em}\cdot}\,}
\def\SP{\,{>\hspace{-.3em}\cdot}\,}
\def\PM{\,{\cdot\hspace{-.3em}<}\,}
\def\PS{\,{\cdot\hspace{-.3em}>}\,}
\def\EP{\,{=\hspace{-.2em}\cdot}\,}
\def\PP{\,{+\hspace{-.1em}\cdot}\,}
\def\PE{\,{\cdot\hspace{-.2em}=}\,}
\def\N{\mathbb N}
\def\C{\mathbb C}
\def\Q{\mathbb Q}
\def\R{\mathbb R}
\def\T{\mathbb T}
\def\A{\mathbb A}
\def\Z{\mathbb Z}
\def\demi{\frac{1}{2}}

\begin{titlepage}
\author{Abed Bounemoura\footnote{CNRS - CEREMADE, Université Paris Dauphine \& IMCCE, Observatoire de Paris} {} and St\'ephane Fischler\footnote{Laboratoire de math\'ematiques d'Orsay, Univ Paris Sud, 91405 Orsay Cedex, France}}
\title{\LARGE{\textbf{The classical KAM theorem for Hamiltonian systems via rational approximations}}}
\end{titlepage}

\maketitle

\begin{abstract}
In this paper, we give a new proof of the classical KAM theorem on the persistence of an invariant quasi-periodic torus, whose frequency vector satisfies the Bruno-Rüssmann condition, in real-analytic non-degenerate Hamiltonian systems close to integrable. The proof, which uses rational approximations instead of small divisors estimates, is an adaptation to the Hamiltonian setting of the method we introduced in \cite{BF12} for perturbations of constant vector fields on the torus. 
\end{abstract}

\section{Introduction}\label{s1}

In this paper, we consider small perturbations of integrable Hamiltonian systems, which are defined by a Hamiltonian function of the form
\[ H(p,q)=h(p)+\epsilon f(p,q), \quad (p,q) \in \R^n \times \T^n, \quad 0 \leq \epsilon <1, \]
where $n\geq 2$ is an integer and $\T^n=\R^n / \Z^n$: the Hamiltonian system associated to this Hamiltonian function is then given by
\begin{equation*}
\begin{cases}
\dot{p}=- \partial_q H(p,q)=-\epsilon \partial_q f(p,q),  \\
\dot{q}=\partial_p H(p,q)=\nabla h (p)+\epsilon \partial_p f(p,q).
\end{cases}
\end{equation*}
When $\epsilon=0$, the system associated to $H=h$ is trivially integrable: all solutions are given by
\[ (p(t),q(t))=(p(0),q(0)+t \nabla h (p(0)) \; [\Z^n]), \] 
and therefore the sets $\mathcal{T}_{p_0}=\{p_0\} \times \T^n$, $p_0 \in \R^n$, are invariant tori on which the dynamics is quasi-periodic with frequency $\omega_0=\nabla h(p_0) \in \R^n$.

Now for $\epsilon >0$, the system is in general no longer integrable, and one is interested to know whether such quasi-periodic solutions persist under an arbitrary small perturbation. It is not hard to see that if the frequency $\omega_0 \in \R^n$ is resonant, that is if there exists a vector $k\in \Z^n\setminus \{0\}$ such that $k\cdot \omega_0=0$, then the torus $\mathcal{T}_{p_0}$ is destroyed by a general perturbation. This goes back to Poincaré, who initiated the study of small perturbations of integrable Hamiltonian systems in his seminal work on Celestial Mechanics.

\subsection{K for Kolmogorov}\label{s11}

The fate of quasi-periodic solutions with non-resonant frequencies remained an open question for more than half a century, until it was solved by Kolmogorov. In \cite{Kol54}, he proved that if $H$ is real-analytic and if $h$ is non-degenerate at some point $p_0 \in \R^n$ in the sense that $\nabla h$ is a local diffeomorphism at $p_0$, then, provided $\omega_0=\nabla h (p_0)$ satisfies a classical Diophantine condition, the torus $\mathcal{T}_{p_0}$ survives an arbitrary small perturbation. The condition on $\omega_0$ is a strengthening of the non-resonance condition, namely one requires the existence of constants $\gamma>0$ and $\tau \geq n-1$ such that for all $k\in \Z^n\setminus\{0\}$,
\[ |k \cdot \omega_0| \geq \gamma |k|_1^{-\tau}, \] 
where $|k|_1=|k_1|+\cdots+|k_n|$ (this is a generalization of a condition introduced by Siegel in \cite{Sie42} for the problem of linearization of a one-dimensional holomorphic map at an elliptic fixed point). Whereas classical Hamiltonian perturbation theory, as pioneered by Poincaré, only yields formal quasi-periodic solutions through an iterative procedure which may or may not converge, Kolmogorov's main idea was to focus on such a strongly non-resonant torus in order to use a modified and rapidly converging inductive scheme, similar to a Newton method, leading to the persistence of this torus. 

\subsection{A and M for Arnold and Moser}\label{s12}

Kolmogorov's fundamental theorem was later revisited and improved by Arnold and Moser, leading to what is known as KAM theory.

In \cite{Arn63a}, Arnold gave a more detailed and technically different proof, under a different non-degeneracy assumption on the integrable Hamiltonian, and in \cite{Arn63b}, he further improved the non-degeneracy assumption in order to apply the theorem to Celestial Mechanics.

In the meantime, in \cite{Mos62}, Moser was able to replace the analyticity condition on the Hamiltonian by a mere finite differentiability condition, in the related context of invariant curves of area-preserving maps of the annulus. Moreover, in \cite{Mos67}, he introduced a powerful formalism for the perturbation theory (not necessarily Hamiltonian) of quasi-periodic solutions, which eventually led to many applications (see \cite{BHS96} for instance).

\subsection{R for Rüssmann}\label{s13}

Further important contributions to KAM theory are due to Rüssmann. Indeed, following works of Arnold and Pyartli, Rüssmann was able to find the most general non-degeneracy condition for the integrable Hamiltonian $h$: in the analytic case, it is sufficient to require that locally, the image of the gradient map $\nabla h$ is not contained in any hyperplane of $\R^n$ (it is also necessary as shown by Sevryuk in \cite{Sev95}). This was announced in \cite{Rus89}, and details are given in \cite{Rus01}. Also, he was able to greatly relax the condition imposed on the frequency, going beyond the classical Diophantine condition (see \cite{Rus94}, \cite{Rus01}). This condition, which generalizes a condition introduced by Bruno in the context of holomorphic linearization (\cite{Bru71},\cite{Bru72}), is now usually called Bruno-Rüssmann condition (see \S\ref{s22} for a definition), and is known to be optimal in one-dimensional problems following works of Yoccoz (see \cite{Yoc95}, \cite{Yoc02}). This extension to more general frequency vectors also led to a different method of proof, in which no rapid convergence is involved (\cite{Rus94}, \cite{Rus01}, see also \cite{Rus10} for the latest improvement of this method).  

\subsection{Other approaches to the classical KAM theorem}\label{s14}

Apart from Rüssmann's modified iterative scheme, a number of other proofs have appeared in the literature.

First, the classical iterative scheme of Kolmogorov has been replaced by the use of an adapted implicit function theorem in a scale of Banach spaces (or in a Fréchet space), following works of Zehnder (\cite{Zeh75}, \cite{Zeh76}), Herman (\cite{Bos86}) and more recently Féjoz (\cite{Fej12}).

Another proof, based on the Lagrangian formalism and that avoids any coordinate transformation, was presented in \cite{LM01} in the case of invariant curves for area-preserving maps of the annulus, and in \cite{SZ89} for Hamiltonian systems in any number of degrees of freedom. 

But perhaps the most striking proof is due to Eliasson. The classical theorem of Kolmogorov shows, a posteriori and in an indirect way, that some formal solutions of classical perturbation theory do converge. In \cite{Eli96}, Eliasson managed to prove directly the convergence of these formal solutions, by adding suitable terms in the formal series in order to exhibit subtle cancellations yielding the absolute convergence.  

At last, we should also point out that using a multi-dimensional continued fraction algorithm due to Lagarias, Khanin, Lopes Dias and Marklof gave a proof of the KAM theorem with techniques closer to renormalization theory (see \cite{KDM07} for the case of constant vector fields on the torus, and \cite{KDM06} for the case of Hamiltonian systems).

\subsection{Approach via rational approximations}\label{s15}

The purpose of this article, which can be considered as a continuation of our previous work \cite{BF12}, is to present yet another proof of the classical KAM theorem for Hamiltonian systems, which differs qualitatively from all other existing proofs as it does not involve any small divisors estimates and Fourier series expansions.

First we should recall that in the classical approach to the KAM theorem, as well as the other methods of proof we just mentioned, a central role is played by the following equation:
\begin{equation}\label{eqlin}
\mathcal{L}_{\omega_0} g=f-[f], \quad [f]=\int_{\T^n}f(\theta)d\theta,
\end{equation}
where $g$ (respectively $f$) is the unknown (respectively known) smooth function on $\T^n$, $\omega_0$ is non-resonant and $\mathcal{L}_{\omega_0}$ is the derivative in the direction of $\omega_0$. It is precisely in trying to solve Equation~\eqref{eqlin} that small divisors arise: geometrically, one needs to integrate along the integral curves $t \mapsto \theta +t \omega_0 \, [\Z^n]$, and these curves are not closed (they densely fill the torus). Analytically, one needs to invert the operator $\mathcal{L}_{\omega_0}$ acting on the space of smooth functions, and this operator is unbounded. Indeed, this operator can be diagonalized in a Fourier basis: letting $e_k(\theta)=e^{2\pi ik\cdot\theta}$ for $k\in\Z^n$ and expanding in Fourier series $g=\sum_{k\in\Z^n}g_k e_k$ and $f=\sum_{k\in\Z^n}f_k e_k$, the solution is given by
\[ g_0=0, \quad g_k=(2\pi ik\cdot\omega_0)^{-1}f_k, \quad k\in \Z^n\setminus\{0\}.  \]
The quantities $k\cdot\omega_0$, which enter in the denominators, can be arbitrarily small if the norm $|k|_1$ is arbitrarily large: these are called the small divisors, and are the main source of complications.

In this article, we will show how one can, using rational approximations, avoid solving Equation~\eqref{eqlin} and therefore avoid facing small divisors. Assume, without loss of generality, that the first component of $\omega_0$ has been normalized to one. Using a Diophantine result deduced in~\cite{BF12} from classical properties of geometry of numbers, we will approximate $\omega_0$ by $n$ rational vectors $v_1,\dots,v_n \in \Q^n$, with a control on their denominators $q_1,\dots,q_n$ in terms of the quality of the approximation, and such that the integer vectors $q_1v_1,\dots,q_nv_n$ form a $\Z$-basis of $\Z^n$. This result will enable us to replace the study of Equation~\eqref{eqlin} with the study of the equations
\begin{equation}\label{eqlin2}
\mathcal{L}_{v_j} g_j=f_j-[f_j]_{v_j}, \quad [f_j]_{v_j}(\theta)=\int_{0}^{1}f(\theta+tq_jv_j)dt, \quad 1 \leq j \leq n,
\end{equation}
where the $f_j$ are defined inductively by $f_1=f$ and $f_{j+1}=[f_j]_{v_j}$ for $1 \leq j \leq n-1$, and the $g_j$ are the unknown. Equations~\eqref{eqlin2} are much simpler than Equation~\eqref{eqlin}, they can be solved without Fourier expansions by the following simple integral formula
\[ g_j(\theta)=q_j\int_{0}^{1}(f_j-[f_j]_{v_j})(\theta+tq_jv_j)tdt \]
and there are no small divisors: geometrically, the integral curves $t \mapsto \theta +tv_j \, [\Z^n]$ are $q_j$-periodic hence closed, and analytically, the inverse operator of $\mathcal{L}_{v_j}$ is bounded (by $q_j$, with respect to any translation-invariant norm on the space of functions on the torus).   

In \cite{BF12}, this approach was already used in the model problem of perturbations of constant vector fields on the torus; the aim of this article is therefore to explain how to adapt the arguments of~\cite{BF12} to the context of real-analytic non-degenerate Hamiltonian systems close to integrable.

\bigskip

\noindent{\bf Acknowledgements: } The first author would like to thank the organizers of the conference ``Geometry, Dynamics, Integrable Systems - GDIS 2013" in Izhevsk, Russia, for the invitation to give a talk on this topic, and IMPA, Rio de Janeiro, Brazil, for its hospitality during the preparation of this work. The second author is partially supported by Agence Nationale de la Recherche (project HAMOT, ref. ANR 2010 BLAN-0115), and thanks Michel Waldschmidt for his advice.

\section{Statements}\label{s2}

As explained in the Introduction, the result that we will prove here is not new, only the method of proof is. For convenience, we will follow the very nice survey \cite{Pos01} for the exposition of the statements. Compared to \cite{Pos01}, we decided for simplicity to focus on the persistence of a single invariant torus instead of a Cantor family of invariant tori; on the other hand, our frequency will be assumed to satisfy the Bruno-Rüssmann condition, which is more general than the classical Diophantine condition.

\subsection{Setting}\label{s21}

Recall that $n\geq 2$ is an integer, $\T^n=\R^n / \Z^n$ and let $D \subseteq \R^n$ be an open domain containing the origin. For a small parameter $\epsilon\geq 0$, we consider a Hamiltonian function $H : D \times \T^n  \rightarrow \R$ of the form
\begin{equation}\label{Ham1}
\begin{cases}
H(p,q)=h(p)+\epsilon f(p,q),  \\
\nabla h (0)=\omega_0=(1,\bar{\omega}_0) \in \R^n, \quad \bar{\omega}_0 \in [-1,1]^{n-1}. 
\end{cases}\tag{$*$}
\end{equation}
Assuming that the vector $\nabla h (0)=\omega_0$ is non-zero, it can always be written as above, re-ordering its components and re-scaling the Hamiltonian if necessary. The integrable Hamiltonian $h$ is said to be \textit{non-degenerate} at the origin if the map $\nabla h : D \rightarrow \R^n$ is a local diffeomorphism at the origin. Upon restricting $D$ if necessary, we may assume that $\nabla h$ is actually a global diffeomorphism. The Hamiltonian $H$ is said to be \textit{real-analytic} on $\bar{D} \times \T^n$, where $\bar{D}$ denotes the closure of $D$ in $\R^n$, if it is analytic on a fixed (that is, independent of $\epsilon$) neighborhood of $\bar{D} \times \T^n$ in $\R^n \times \T^n$. This implies that $H$ can be extended as a holomorphic function on a fixed complex neighborhood of $\bar{D} \times \T^n$ in $\C^n \times \T_{\C}^n$, where $\T_{\C}^n=\C^n / \Z^n$, which is real-valued for real arguments.  

\subsection{Bruno-Rüssmann condition}\label{s22}

For $Q\geq 1$, let us define the function $\Psi=\Psi_{\omega_0}$ by
\begin{equation}\label{eqpsi}
\Psi(Q)=\sup\{|k\cdot\omega_0|^{-1}\; | \; k \in \Z^n, \; 0 < |k|_1\leq Q\} \in [1,+\infty].
\end{equation}
Recall that the vector $\omega_0$ is said to be non-resonant if $k\cdot \omega_0=0$ implies $k=0 \in \Z^n$, which is equivalent to $\Psi(Q)$ being finite for all $Q\geq 1$. We say that $\omega_0$ satisfies the Bruno-Rüssmann condition if it is non-resonant and
\begin{equation}\label{BR}
\int_{1}^{+\infty}Q^{-2}\ln(\Psi(Q))dQ < \infty. \tag{BR}
\end{equation}

Now let us define two other functions $\Delta=\Delta_{\omega_0}$ and $\Delta^*=\Delta^*_{\omega_0}$ by $\Delta(Q)=Q\Psi(Q)$ for $Q \geq 1$, and $\Delta^*(x)=\sup\{Q \geq 1\; | \; \Delta(Q)\leq x\}$ for $x \geq \Delta(1)$. We obviously have $\Delta^*(\Delta(Q))=Q$ and $\Delta(\Delta^*(x)) \leq x$. Moreover, it is not hard to check (see \cite{BF12}, Appendix A) that $\Psi$ satisfies~\eqref{BR} if and only if $\Delta^*$ satisfies
\begin{equation}\label{BR2}
\int_{\Delta (1)}^{+\infty}(x\Delta^*(x))^{-1}dx < \infty.
\end{equation}

\subsection{Classical KAM theorem}\label{s23}

Consider the map $\Theta_0 : \T^n \rightarrow D \times \T^n$ given by $\Theta_0(q)=(0,q)$: this is a real-analytic torus embedding such that $\Theta_0(\T^n)$ is invariant by the Hamiltonian flow of $H_0$ and quasi-periodic with frequency $\omega_0$. The classical KAM theorem states that this invariant quasi-periodic torus is preserved, being only slightly deformed, by an arbitrary small perturbation, provided $h$ is non-degenerate, $H$ analytic and $\omega_0$ satisfies the Bruno-Rüssmann condition. 

\begin{theorem}\label{thm1}
Let $H$ be as in~\eqref{Ham1}, with $h$ non-degenerate and $H$ real-analytic, and assume that $\Psi=\Psi_{\omega_0}$ satisfies~\eqref{BR}. For $\epsilon$ small enough, there exists a real-analytic torus embedding $\Theta_{\omega_0} : \T^n \rightarrow D \times \T^n$ such that $\Theta_{\omega_0}(\T^n)$ is invariant by the Hamiltonian flow of $H$ and quasi-periodic with frequency $\omega_0$. 

Moreover, $\Theta_{\omega_0}$ converges uniformly to $\Theta_0$ as $\epsilon$ goes to zero.
\end{theorem}

As in \cite{Pos01}, Theorem~\ref{thm1} will be deduced from a KAM theorem for a Hamiltonian ``with parameters", for which a quantitative statement is given below.

\subsection{KAM theorem with parameters}\label{s24}

Let us now consider a different setting. Given $r,s,h$ real numbers such that $0 \leq r \leq 1$, $0 \leq s \leq 1$, $0 \leq h \leq 1$, we let
\[ D_{r,s}=\{ I \in \C^n \; | \; |I|<r \} \times \{ \theta \in \T_{\C}^n=\C^n / \Z^n \; | \; |\mathrm{Im}(\theta)|<s \} \]
and 
\[ O_h=\{ \omega \in \C^n \; | \; |\omega-\omega_0|<h \}\]
be complex neighborhoods of respectively $\{0\} \times \T^n$ and $\omega_0$, where $|\,.\,|$ stands for the supremum norm of vectors.

For a small parameter $\varepsilon\geq 0$, consider a function $H$, which is bounded and real-analytic on $D_{r,s}\times O_h$, and of the form 
\begin{equation}\label{Ham2}
\begin{cases}
H(I,\theta,\omega)=N(I,\omega)+P(I,\theta,\omega), \\
N(I,\omega)=e(\omega)+\omega \cdot I, \quad |P|_{r,s,h}\leq\varepsilon,
\end{cases}
\tag{$**$}
\end{equation}
where 
\[|P|_{r,s,h}=\sup_{(I,\theta,\omega) \in D_{r,s}\times O_h}|P(I,\theta,\omega)|.\]
The function $H$ should be considered as a real-analytic Hamiltonian on $D_{r,s}$, depending analytically on a parameter $\omega \in O_h$; for a fixed parameter $\omega \in O_h$, when convenient, we will write 
\[ H_\omega (I,\theta)=H(I,\theta,\omega), \quad N_\omega(I)=N(I,\omega), \quad P_\omega(I,\theta)=P(I,\theta,\omega). \] 

Let $B=\{I \in \R^n \; | \; |I|<r\}$ so that $B \times \T^n$ is the real part of the domain $D_{r,s}$, and $\Phi_0 : \T^n \rightarrow B \times \T^n$ be the map given by $\Phi_0(\theta)=(0,\theta)$. Then $\Phi_0(\T^n)$ is an embedded real-analytic torus in $B \times \T^n$, invariant by the Hamiltonian flow of $N_{\omega_0}$ and quasi-periodic with frequency $\omega_0$. The next theorem states that this quasi-periodic torus will persist, being only slightly deformed, as an invariant torus not for the Hamiltonian flow of $H_{\omega_0}$ but for the Hamiltonian flow of $H_{\tilde{\omega}}$, where $\tilde{\omega}$ is a parameter close to $\omega_0$, provided $\varepsilon$ is sufficiently small and $\omega_0$ satisfies the Bruno-Rüssmann condition. Here is a more precise statement.

\begin{theorem}\label{thm2}
Let $H$ be as in~\eqref{Ham2}, with $\Delta^*=\Delta^*_{\omega_0}$ satisfying~\eqref{BR2}. Then there exist positive constants $c_1 \leq 1$, $c_2\leq 1$, $c_3 \leq 1$, $c_4 \geq 1$ and $c_5 \geq 1$ depending only on $n$ such that if
\begin{equation}\label{cond0}
\varepsilon r^{-1} \leq c_1 h \leq c_2 (Q_0\Psi(Q_0))^{-1}
\end{equation}
where $Q_0 \geq 1$ is sufficiently large so that
\begin{equation}\label{eqQ0}
Q_0^{-1}+(\ln 2)^{-1}\int_{\Delta(Q_0)}^{+\infty}\frac{dx}{x\Delta^*(x)} \leq c_3 s,
\end{equation}
there exist a real-analytic torus embedding $\Phi_{\omega_0} : \T^n \rightarrow B \times \T^n$ and a vector $\tilde{\omega} \in \R^n$ such that $\Phi_{\omega_0}(\T^n)$ is invariant by the Hamiltonian flow of $H_{\tilde{\omega}}$ and quasi-periodic with frequency $\omega_0$. Moreover, $\Phi_{\omega_0}$ is real-analytic on $\T^n_{s/2}=\{ \theta \in \T_{\C}^n \; | \; |\mathrm{Im}(\theta)|<s/2 \}$ and we have the estimates
\begin{equation}\label{estim0}
|W(\Phi_{\omega_0}-\Phi_0)|_{s/2} \leq c_4 \varepsilon (rh)^{-1}, \quad |\tilde{\omega}-\omega_0| \leq c_5\varepsilon r^{-1},
\end{equation}
where $W=\mathrm{Diag}(r^{-1}\mathrm{Id},Q_0^{-1}\mathrm{Id})$.
\end{theorem}

Theorem~\ref{thm1} follows directly from Theorem~\ref{thm2}, introducing the frequencies $\omega = \nabla h (p)$ as independent parameters and compensating the shift of frequency $\tilde{\omega}-\omega_0$ using the non-degeneracy assumption on $h$. This deduction is classical but for completeness we will repeat the details below, following \cite{Pos01}.

\section{Proof of Theorem~\ref{thm1} assuming Theorem~\ref{thm2}}\label{s3}

In this section, we assume Theorem~\ref{thm2} and we show how it implies Theorem~\ref{thm1}. 

\begin{proof}[Proof of Theorem~\ref{thm1}]
For $p_0 \in D$, we expand $h$ in a small neighborhood of $p_0$: writing $p=p_0+I$ for $I$ close to zero, we get
\[ h(p)=h(p_0)+ \nabla h(p_0)\cdot I + \int_{0}^{1}(1-t)\nabla^2 h(p_0 +tI)I\cdot I dt. \]
Since $\nabla h : D \rightarrow \Omega$ is a diffeomorphism, instead of $p_0$ we can use $\omega=\nabla h(p_0)$ as a new variable, and we write
\[ h(p)=e(\omega)+ \omega\cdot I + P_h(I,\omega) \] 
with $e(\omega)=h(\nabla^{-1}(\omega))$ and $P_h(I,\omega)=\int_{0}^{1}(1-t)\nabla^2 h(\nabla^{-1}(\omega) +tI)I\cdot I dt$. Letting $\theta=q$ and 
\[ P_\epsilon(I,\theta,\omega)=\epsilon f(p,q)=\epsilon f(p_0+I,\theta)=\epsilon f(\nabla^{-1}(\omega)+I,\theta), \]
we can eventually define
\[ H(I,\theta,\omega)= H(p,q)= e(\omega)+ \omega\cdot I + P_h(I,\omega)+ P_\epsilon(I,\theta,\omega)= e(\omega)+ \omega\cdot I + P(I,\theta,\omega).\]
This Hamiltonian is obviously real-analytic in $(I,\theta,\omega)$, hence we can fix some small $0<s<1$ and $0<h<1$ so that $P$ is real-analytic on the complex domain $D_{r,s} \times O_h$ for all sufficiently small $0<r<1$. Moreover, as $\Psi=\Psi_{\omega_0}$ satisfies~\eqref{BR}, $\Delta^*=\Delta^*_{\omega_0}$ satisfies~\eqref{BR2}, and choosing $Q_0=Q_0(s)$ sufficiently large so that~\eqref{eqQ0} is satisfied, we may assume, restricting $h$ if necessary, that the second inequality of~\eqref{cond0} holds true, namely $c_1 h \leq c_2 (Q_0\Psi(Q_0))^{-1}$. In the same way, we may also assume that the real part of $O_h$ is contained in $\Omega$.

Now since $P(I,\theta,\omega)=P_h(I,\omega)+ P_\epsilon(I,\theta,\omega)$, we have
\[ |P|_{r,s,h} \leq \varepsilon=Mr^2 +F\epsilon \]
where
\[ M=\sup_{p \in D} |\nabla^2 h(p)|, \quad F=\sup_{(p,q) \in D \times \T^n} |f(p,q)|. \]
Therefore we choose $r=(M^{-1}F\epsilon)^{1/2}$ to have $\varepsilon=2F\epsilon$, and assuming $\epsilon \leq (4MF)^{-1} c_1^2 h^2$, we have
\[ \varepsilon r^{-1} = 2F\epsilon (MF^{-1}\epsilon^{-1})^{1/2}=2(MF\epsilon)^{1/2}  \leq c_1 h,\] 
hence the first inequality of~\eqref{cond0} is satisfied. So Theorem~\ref{thm2} can be applied: there exist a real-analytic torus embedding $\Phi_{\omega_0} : \T^n \rightarrow B \times \T^n$ and a vector $\tilde{\omega} \in \R^n$ such that $\Phi_{\omega_0}(\T^n)$ is invariant by the Hamiltonian flow of $H_{\tilde{\omega}}$ and quasi-periodic with frequency $\omega_0$. Moreover, $\Phi_{\omega_0}$ is real-analytic on $\T^n_{s/2}=\{ \theta \in \T_{\C}^n \; | \; |\mathrm{Im}(\theta)|<s/2 \}$ and we have the estimates
\begin{equation}\label{estim0f}
|W(\Phi_{\omega_0}-\Phi_0)|_{s/2} \leq c_4 \varepsilon (rh)^{-1}, \quad |\tilde{\omega}-\omega_0| \leq c_5\varepsilon r^{-1},
\end{equation}
where $W=\mathrm{Diag}(r^{-1}\mathrm{Id},Q_0^{-1}\mathrm{Id})$.

Since $\tilde{\omega}$ is real and the real part of $O_h$ is contained in $\Omega$, there exists $\tilde{I}$ close to zero such that $\nabla h (\tilde{I})=\tilde{\omega}$. Now observe that an orbit $(I(t),\theta(t))$ for the Hamiltonian $H_{\tilde{\omega}}$ corresponds to an orbit $(p(t),q(t))=(\tilde{I}+I(t),\theta(t))$ for our original Hamiltonian. Hence, if we define $T : B \times \T^n \rightarrow D \times \T^n$ by $T(I,\theta)=(\tilde{I}+I,\theta)$ and 
\[ \Theta_{\omega_0} = T \circ \Phi_{\omega_0} : \T^n \rightarrow D \times \T^n, \]
then $\Theta_{\omega_0}$ is a real-analytic torus embedding such that $\Theta_{\omega_0}(\T^n)$ is invariant by the Hamiltonian flow of $H$ and quasi-periodic with frequency $\omega_0$. 

Moreover, as $\epsilon$ goes to zero, it follows from~\eqref{estim0f} that $\Phi_{\omega_0}$ converges uniformly to $\Phi_0$ and $\tilde{\omega}$ converges to $\omega_0$, hence $T$ converges uniformly to the identity which eventually implies that $\Theta_{\omega_0}$ converges uniformly to $\Theta_0$.

\end{proof}

\section{Proof of Theorem~\ref{thm2}}\label{s4}

This section is devoted to the proof of Theorem~\ref{thm2}, in which we will construct, by an iterative procedure, a vector $\tilde{\omega}$ close to $\omega_0$ and a real-analytic torus embedding $\Phi_{\omega_0}$ whose image is invariant by the Hamiltonian flow of $H_{\tilde{\omega}}$. 
We start, in \S\ref{s41}, by recalling the Diophantine result of~\cite{BF12} which will be crucial in our approach. Then, in \S\ref{s42}, we describe an elementary step of our iterative procedure, and finally, in \S\ref{s43}, we will show one can perform infinitely many steps to obtain a convergent scheme.

In this section, for simplicity, we shall adopt the following notation: given positive real numbers $u$ and $v$, we will write $u \MP v$ (respectively $u \PM v$) if, for some constant $C\geq 1$ which depends only on $n$ and could be made explicit, we have $u\leq Cv$ (respectively $Cu \leq v$).

\subsection{Approximation by rational vectors}\label{s41}

Recall that we have written $\omega_0=(1,\bar{\omega}_0) \in \R^n$ with $\bar{\omega}_0 \in [-1,1]^{n-1}$. For a given $Q\geq 1$, it is always possible to find a rational vector $v=(1,p/q) \in \Q^n$, with $p \in \Z^{n-1}$ and $q \in \N$, which is a $Q$-approximation in the sense that $|q\omega_0 - qv|\leq Q^{-1}$, and for which the denominator $q$ satisfies the upper bound $q \leq Q^{n-1}$: this is essentially the content of Dirichlet's theorem on simultaneous rational approximations, and it holds true without any assumption on $\omega_0$. In our situation, since we have assumed that $\omega_0$ is non-resonant, it is not hard to see that there exist not only one, but $n$ linearly independent rational vectors in $\Q^n$ which are $Q$-approximations. Moreover, one can obtain not only linearly independent vectors, but rational vectors $v_1,\dots,v_n$ of denominators $q_1, \dots,q_n$ such that the associated integer vectors $q_1v_1,\dots,q_nv_n$ form a $\Z$-basis of $\Z^n$. However, the upper bound on the corresponding denominators $q_1, \dots, q_n$ is necessarily larger than $Q^{n-1}$, and is given by a function of $Q$ that we can call here $\Psi'=\Psi_{\omega_0}'$ (see \cite{BF12} for more precise and general information, but note that in this reference, $\Psi'$ was denoted by $\Psi$ and $\Psi$, which we defined in~\eqref{eqpsi}, was denoted by $\Psi'$). A consequence of the main Diophantine result of \cite{BF12} is that this function $\Psi'$ is in fact essentially equivalent to the function $\Psi$. 

\begin{proposition}\label{dio}
Let $\omega_0=(1,\bar{\omega}_0) \in \R^n$ be a non-resonant vector, with $\bar{\omega}_0 \in [-1,1]^{n-1}$. For any $1 \PM Q$, there exist $n$ rational vectors $v_1, \dots, v_n$, of denominators $q_1, \dots, q_n$, such that $q_1v_1, \dots, q_nv_n$ form a $\Z$-basis of $\Z^n$ and for $j\in\{1,\dots,n\}$,
\[ |\omega_0-v_j|\MP(q_j Q)^{-1}, \quad 1 \leq q_j \MP \Psi(Q).\]
\end{proposition}

For a proof, we refer to \cite{BF12}, Theorem 2.1 and Proposition $2.3$. 

Now given a $q$-rational vector $v$ and a function $P$ defined on $D_{r,s}\times O_h$, we define
\[ [P]_v(I,\theta,\omega)= \int^{1}_{0} P(I,\theta+tqv,\omega)dt.\]
Given $n$ rational vectors $v_1,\dots,v_n$, we let $[P]_{v_1,\dots,v_d}=[\cdots[P]_{v_1}\cdots]_{v_d}$. Finally we define
\[ [P](I,\omega)=\int_{\T^n}P(I,\theta,\omega)d\theta. \]
A consequence of the fact that the vectors $q_1v_1, \dots, q_nv_n$ form a $\Z$-basis of $\Z^n$ is contained in the following proposition. 

\begin{proposition}\label{cordio}
Let $v_1,\dots,v_n$ be rational vectors, of denominators $q_1,\dots,q_n$, such that $q_1v_1, \dots, q_nv_n$ form a $\Z$-basis of $\Z^n$, and $P$ a function defined on $D_{r,s}\times O_h$. Then
\[ [P]_{v_1,\dots,v_n}=[P]. \]  
\end{proposition} 

A proof of this proposition can be found in \cite{Bou13}, Corollary 6.

\subsection{KAM step}\label{s42}

Now we describe an elementary step of our iterative procedure. To our Hamiltonian $H$, we will apply transformations of the form
\[ \mathcal{F}=(\Phi,\varphi): (I,\theta,\omega) \mapsto (\Phi(I,\theta,\omega),\varphi(\omega)) \]
which consist of a parameter-depending change of coordinates $\Phi$ and a change of parameters $\varphi$. Moreover, our change of coordinates will be of the form $\Phi=(U,V)$, where $U$ is affine in $I$ and $V$ is independent of $I$ , and will be symplectic for each fixed parameter $\omega$. It is easy to check that such transformations $\mathcal{F}=(\Phi,\varphi)$ form a group under composition.

From now on, we fix a positive constant $\eta$ sufficiently small (one could take, for instance, $\eta=1/66$).

\begin{proposition}\label{kamstep}
Let $H$ be as in~\eqref{Ham2}, with $\omega_0=(1,\bar{\omega}_0) \in \R^n$ non-resonant, consider $0<\sigma< s$ and $1 \PM Q$, and assume that
\begin{equation}\label{cond1}
\varepsilon r^{-1} \PM h \PM (Q\Psi(Q))^{-1}, \quad  1 \PM Q\sigma.
\end{equation}
Then there exists a real-analytic transformation 
\[ \mathcal{F}=(\Phi,\varphi) : D_{\eta r, s-\sigma} \times O_{h/4} \rightarrow D_{r, s} \times O_{h}, \]
such that $H \circ \mathcal{F}=N^+ + P^+$ with $N^+(I,\omega)=e^+(\omega)+\omega\cdot I$ and $|P^+|_{\eta r, s-\sigma,h/4} \leq \eta\varepsilon/8$. Moreover, we have the estimates
\[ |W(\Phi-\mathrm{Id})|_{\eta r,s-\sigma,h} \MP (r\sigma)^{-1}\Psi(Q)\varepsilon, \quad |W(D\Phi-\mathrm{Id})W^{-1}|_{\eta r,s-\sigma,h} \MP (r\sigma)^{-1}\Psi(Q)\varepsilon, \]
\[ |\varphi-\mathrm{Id}|_{h/4} \MP \varepsilon r^{-1}, \quad   h|D\varphi-\mathrm{Id}|_{h/4} \MP \varepsilon r^{-1},\]
where $W=\mathrm{Diag}(r^{-1}\mathrm{Id},\sigma^{-1}\mathrm{Id})$.
\end{proposition}

\begin{proof}

We divide the proof of the KAM step in six small steps. Except for the last one, everything will be uniform in $\omega \in O_h$, so for simplicity, in the first five steps we will drop the dependence on the parameter $\omega \in O_h$. Let us first notice that~\eqref{cond1} implies the following five inequalities:
\begin{equation}\label{condcond}
h \PM (Q\Psi(Q))^{-1}, \quad \varepsilon \PM \sigma r\Psi(Q)^{-1}, \quad \varepsilon \PM r(Q\Psi(Q))^{-1}, \quad 1 \PM Q\sigma, \quad \varepsilon r^{-1} \PM h. 
\end{equation}

\medskip

\textit{1. Affine approximation of $P$}

\medskip

Let $\bar{P}$ be the linearization of $P$ in $I$ at $I=0$; that is,  
\[ \bar{P}(I,\theta)=P(0,\theta)+\partial_I P(0,\theta) \cdot I.\] Using Cauchy's estimate, it is easy to see that $|\bar{P}|_{r,s} \MP \varepsilon$. Moreover, using Lemma~\ref{tech1}, we have
\begin{equation}\label{trunc}
|P-\bar{P}|_{2\eta r,s} \leq (2\eta)^2(1-2\eta)^{-1}\varepsilon \leq \eta\varepsilon/16
\end{equation}
where we used in the second inequality that $\eta$ is small enough.

\medskip

\textit{2. Rational approximations of $\omega_0$}

\medskip

Since $\omega_0$ is non-resonant, given $1 \PM Q$, we can apply Proposition~\eqref{dio}: there exist $n$ rational vectors $v_1, \dots, v_n$, of denominators $q_1, \dots, q_n$, such that $q_1v_1, \dots, q_nv_n$ form a $\Z$-basis of $\Z^n$ and for $j\in\{1,\dots,n\}$,
\[ |\omega_0-v_j|\MP(q_j Q)^{-1}, \quad 1 \leq q_j \MP \Psi(Q).\]

\medskip

\textit{3. Rational approximations of $\omega \in O_h$}

\medskip

For any $\omega \in O_h$, using the first inequality in~\eqref{condcond} and $q_j \MP \Psi(Q)$,  we have
\begin{equation}\label{shift}
|\omega-v_j| \leq |\omega-\omega_0|+|\omega_0-v_j| \MP h + (q_j Q)^{-1} \MP (Q\Psi(Q))^{-1} + (q_j Q)^{-1} \MP (q_j Q)^{-1}.
\end{equation}

\medskip

\textit{4. Successive rational averagings}

\medskip

Let $P_1=\bar{P}$, and define inductively $P_{j+1}=[P_j]_{v_j}$ for $1 \leq j \leq n$. Let us also define $F_j$, for $1 \leq j \leq n$, by
\[ F_j(I,\theta)=q_j \int_0^1 (P_j-P_{j+1})(I,\theta+tq_jv_j)tdt \]
and $N_j$ by $N_j(I)=e(\omega)+v_j \cdot I$. It is then easy to check, by a simple integration by parts, that the equations
\begin{equation}\label{eqhomo}
\{F_j,N_j\}=P_j-P_{j+1}, \quad 1 \leq j \leq n,
\end{equation}
are satisfied, where $\{\,.\,,\,.\,\}$ denotes the usual Poisson bracket. Moreover, we obviously have the estimates
\begin{equation}\label{Estim0}
|P_j|_{r,s} \leq |\bar{P}|_{r,s} \MP \varepsilon
\end{equation}
and \begin{equation}\label{estim00}
|F_j|_{r,s} \leq q_j|P_j|_{r,s} \MP \Psi(Q)\varepsilon.
\end{equation}

Next, for any $0 \leq j \leq n$, define $r_j=r-(2n)^{-1}jr$ and $s_j=s-n^{-1}j\sigma$. Obviously $r_j>0$ whereas $s_j>0$ follows from $\sigma < s$. Using~\eqref{estim00} and Cauchy's estimate,
we have
\[ |\partial_\theta F_j|_{r_j,s_j} \leq n(j\sigma)^{-1}|F_j|_{r,s} \MP \sigma^{-1}\Psi(Q)\varepsilon, \quad |\partial_I F_j|_{r_j,s_j} \leq 2n(jr)^{-1}|F_j|_{r,s} \MP r^{-1}\Psi(Q)\varepsilon, \]
and together with the second inequality of~\eqref{cond1} with a suitable implicit constant, we can ensure that 
\[ |\partial_\theta F_j|_{r_j,s_j} \leq (2n)^{-1}r, \quad |\partial_I F_j|_{r_j,s_j} \leq n^{-1}\sigma. \]
This implies that for $1 \leq j \leq n$, the time-one map $X_{F_j}^1$ of the Hamiltonian flow of $F_j$ defines a symplectic real-analytic embedding
\[ X_{F_j}^1=(U_j,V_j) : D_{r_j,s_j} \rightarrow D_{r_{j-1},s_{j-1}}.  \]
Moreover, as $\bar{P}$ is affine in $I$ then so are the $F_j$, for $1 \leq j \leq n$, and therefore $U_j$ is affine in $I$ and $V_j$ is independent of $I$, and we also have the estimates
\begin{equation}\label{estimdist}
|U_j - \mathrm{Id}|_{r_j,s_j} \leq |\partial_\theta F_j|_{r_j,s_j} \MP \sigma^{-1}\Psi(Q)\varepsilon, \quad |V_j - \mathrm{Id}|_{r_j,s_j} \leq |\partial_I F_j|_{r_j,s_j} \MP r^{-1}\Psi(Q)\varepsilon.
\end{equation}
The Jacobian of $X_{F_j}^1$ is given by the matrix
\[ DX_{F_j}^1=
\begin{pmatrix} \partial_I U_j & \partial_\theta U_j  \\ 0 & \partial_\theta V_j  \end{pmatrix}. \]
If we define $r_j^+=r-(4n)^{-1}jr$ and $s_j=s-(2n)^{-1}j\sigma$, then~\eqref{estim00} and Cauchy's estimate also implies that
\[ |\partial_\theta F_j|_{r_j^+,s_j^+} \MP \sigma^{-1}\Psi(Q)\varepsilon, \quad |\partial_I F_j|_{r_j^+,s_j^+} \MP r^{-1}\Psi(Q)\varepsilon, \]
and, as $r_j^+-r_j=(4n)^{-1}jr$ and $s_j^+-s_j=(2n)^{-1}j\sigma$, a further Cauchy's estimate proves that
\[ |\partial_I U_j - \mathrm{Id}|_{r_j,s_j} \MP (r\sigma)^{-1}\Psi(Q)\varepsilon, \quad  |\partial_\theta V_j - \mathrm{Id}|_{r_j,s_j} \MP (r\sigma)^{-1}\Psi(Q)\varepsilon, \quad |\partial_\theta U_j|_{r_j,s_j} \MP \sigma^{-2}\Psi(Q)\varepsilon. \]
The estimates~\eqref{estimdist} and the above estimates can be conveniently written as
\begin{equation}\label{estimdistnew}
|W(X_{F_j}^1-\mathrm{Id})|_{r_j,s_j} \MP (r\sigma)^{-1}\Psi(Q)\varepsilon, \quad |W(DX_{F_j}^1-\mathrm{Id})W^{-1}|_{r_j,s_j} \MP (r\sigma)^{-1}\Psi(Q)\varepsilon, 
\end{equation}
where $W=\mathrm{Diag}(r^{-1}\mathrm{Id},\sigma^{-1}\mathrm{Id})$.

Let $S_j=\omega\cdot I - v_j\cdot I$ so that $N=N_j+S_j$, and let $H_j=N+P_j$. Writing $H_j=N+P_j=N_j+S_j+P_j$ and using the equality~\eqref{eqhomo}, a standard computation based on Taylor's formula with integral remainder gives
\[ H_j \circ X_{F_j}^1 = N+[P_j]_{v_j} +\tilde{P}_j= N+P_{j+1}+\tilde{P}_j \]
with
\[ \tilde{P}_j=\int_0^1 \{ (1-t)P_{j+1}+t P_j +S_j, F_j \} \circ X_{F_j}^t dt  \] 
and where $X_{F_j}^t$ is the time-$t$ map of the Hamiltonian flow of $F_j$. Using~\eqref{shift}, \eqref{Estim0}, \eqref{estim00}, the definition of the Poisson bracket and Cauchy's estimate, on easily obtains 
\begin{equation}\label{estim1}
|\tilde{P}_j|_{r_j,s_j} \MP (\sigma r)^{-1}\Psi(Q)\varepsilon^2 +  (Q\sigma)^{-1}\varepsilon \MP (Q\sigma)^{-1}\varepsilon
\end{equation}
where the last inequality follows from the third inequality of~\eqref{condcond}.

\medskip

\textit{5. Change of coordinates}

\medskip

For $1 \leq j \leq n$, let $\Phi_j= X_{F_1}^1 \circ \cdots \circ X_{F_j}^1 : D_{r_j,s_j} \rightarrow D_{r,s}$. Since $H_1=N+\overline{P}$, by a straightforward induction we have
\[ H_1 \circ \Phi_j = (N+\overline{P}) \circ \Phi_j= N +P_{j+1} +P^+_j    \]
where $P_j^+$ is defined inductively by $P_1^+=\tilde{P}_1$ and $P_{j+1}^+=\tilde{P}_{j+1}+P_j^+ \circ X_{F_{j+1}}^1$ for $1 \leq j \leq n-1$. Let $\Phi=\Phi_n$, and first note that $\Phi=(U,V)$ where $U$ is affine in $I$ and $V$ is independent of $I$, since each $X_{F_j}^1$ is of this form. Moreover, $\Phi : D_{r/4,s-\sigma} \rightarrow D_{r,s}$, so in particular $\Phi : D_{\eta r,s-\sigma} \rightarrow D_{r,s}$ as $\eta$ is small enough, and using~\eqref{estimdistnew} and a classical telescopic argument, we have the estimates
\begin{equation*}
|W(\Phi-\mathrm{Id})|_{\eta r,s-\sigma} \MP (r\sigma)^{-1}\Psi(Q)\varepsilon, \quad |W(D\Phi-\mathrm{Id})W^{-1}|_{\eta r,s-\sigma} \MP (r\sigma)^{-1}\Psi(Q)\varepsilon. 
\end{equation*}
Concerning $P_n^+$, using~\eqref{estim1} and the fourth equality of~\eqref{condcond} with a suitable implicit constant, we can ensure that
\[ |P_n^+|_{\eta r,s-\sigma} \MP (Q\sigma)^{-1}\varepsilon \leq \eta\varepsilon/16.  \] 
Now as $H=N+P=N+\overline{P}+P-\overline{P}=H_1+P-\overline{P}$, this gives
\[ H \circ \Phi =H_1 \circ \Phi + (P - \bar{P}) \circ \Phi= N + P_{n+1} + P_n^+ + (P-\bar{P}) \circ \Phi.  \]
We finally set $P^+=P_n^+ + (P-\bar{P}) \circ \Phi$, and as $P_{n+1}=[\cdots[\bar{P}]_{v_1}\cdots]_{v_n}=[\bar{P}]_{v_1,\dots,v_n}$, by Proposition~\ref{cordio}, $P_{n+1}=[\bar{P}]$, we arrive at
\[ H \circ \Phi = N + [\bar{P}] + P^+.  \]
Using the second inequality of~\eqref{condcond} with an appropriate implicit constant, we may assume that the image of $\Phi$ actually sends $D_{\eta r,s-\sigma}$ into $D_{2\eta r,s}$, and together with~\eqref{trunc}, we obtain the estimate
\[ |P^+|_{\eta r, s-\sigma} \leq |P_n^+|_{\eta r,s-\sigma} + |(P-\bar{P}) \circ \Phi|_{\eta r,s-\sigma} \leq |P_n^+|_{\eta r,s-\sigma} + |P-\bar{P}|_{2\eta r,s-\sigma} \leq \eta\varepsilon/16+\eta\varepsilon/16=\eta\varepsilon/8. \]

\medskip

\textit{6. Change of frequencies}

\medskip

As $\bar{P}$ is affine in $I$, $[\bar{P}]$ is independent of $\theta$ and of the form $[\bar{P}](I,\omega)=c(\omega)+\nu(\omega)\cdot I$, with $c(\omega) \in \C$ and $\nu(\omega) \in \C^n$, and therefore
\[ (N+[\bar{P}])(I,\omega)=e(\omega)+c(\omega)+(\omega+\nu(\omega))\cdot I=e^+(\omega)+(\omega+\nu(\omega))\cdot I. \]
Since $\nu=\partial_I [\bar{P}]$, Cauchy's estimate together with the fifth inequality of~\eqref{condcond} implies that the map $\nu$ satisfy the estimate
\[ |\nu|_{h} \MP \varepsilon r^{-1} \leq h/4. \]
Setting $f(\omega)=\omega+\nu(\omega)$, we can apply Lemma~\ref{tech2} to find a real-analytic inverse $\varphi : O_{h/4} \rightarrow O_h$ to $f$ satisfying the estimates
\[ |\varphi-\mathrm{Id}|_{h/4} \MP \varepsilon r^{-1}, \quad h|D\varphi-\mathrm{Id}|_{h/4} \MP \varepsilon r^{-1}.\]

Eventually, we set $N^+=(N+[\bar{P}]) \circ \varphi$ and $\mathcal{F}=(\Phi,\varphi) : D_{\eta r, s-\sigma} \times O_{h/4} \rightarrow D_{r, s} \times O_{h}$, and we obtain
\[ H \circ \mathcal{F}=N^+ + P^+ \]
as wanted, with the desired estimates on $\mathcal{F}$ and on $P^+$.

\end{proof}

\subsection{Iterations and convergence}\label{s43}

Recall that $\eta$ has been fixed before Proposition~\ref{kamstep}, and now we define, for $i \in \N$, the following decreasing geometric sequences:
\[ \varepsilon_i=(\eta/8)^i \varepsilon, \quad r_i=\eta^i r, \quad h_i=(1/4)^i h.\]
Next, for a constant $Q_0$ to be chosen below, we define $\Delta_i$ and $Q_i$, $i\in \N$, by 
\[ \Delta_i=2^i \Delta(Q_0), \quad Q_i=\Delta^*(\Delta_i)=\sup\{ Q \geq 1 \; | \; \Delta (Q) \leq \Delta_i \}, \]
and then we define $\sigma_i$, $i\in \N$, by
\[ \sigma_i= C Q_i^{-1} \]
where $C \geq 1$ is a constant, depending only $n$, sufficiently large so that the last part of~\eqref{cond1} is satisfied for $\sigma=\sigma_i$ and $Q=Q_i$. Finally, we define $s_i$, $i\in\N$, by $s_0=s$ and $s_{i+1}=s_i-\sigma_i$ for $i\in\N$. 

We claim that, assuming $\Delta^*$ satisfies~\eqref{BR2}, we can choose $Q_0$ sufficiently large so that 
\begin{equation*}
\lim_{i \rightarrow +\infty} s_i\geq s/2 \Longleftrightarrow \sum_{i \in \N} \sigma_i \leq s/2.
\end{equation*}
Indeed, since $Q_i=\Delta^*(\Delta_i)=\Delta^*\left(2^{i}\Delta(Q_0)\right)$, we have
\[ \sum_{i\geq 1}Q_i^{-1}=\sum_{i\geq 1} \frac{1}{\Delta^*\left(2^{i}\Delta(Q_0)\right)}\leq \int_{0}^{+\infty}\frac{dy}{\Delta^*\left(2^{y}\Delta(Q_0)\right)}=(\ln 2)^{-1}\int_{\Delta(Q_0)}^{+\infty}\frac{dx}{x\Delta^*(x)}<+\infty \]
where the last integral converges since $\Delta^*$ satisfies~\eqref{BR2}. Now as $\sigma_i = C Q_i^{-1}$, we have
\[ \sum_{i \in \N} \sigma_i = C \sum_{i \in \N} Q_i^{-1}=CQ_0^{-1}+C \sum_{i \geq 1} Q_i^{-1} \leq CQ_0^{-1}+C (\ln 2)^{-1}\int_{\Delta(Q_0)}^{+\infty}\frac{dx}{x\Delta^*(x)} \leq s/2 \]
provided we choose $Q_0$ sufficiently large in order to have 
\begin{equation}\label{eqQ}
Q_0^{-1}+(\ln 2)^{-1}\int_{\Delta(Q_0)}^{+\infty}\frac{dx}{x\Delta^*(x)} \leq (2C)^{-1} s.
\end{equation}

\begin{proposition}\label{kamiter}
Let $H$ be as in~\eqref{Ham2}, with $\Delta^*=\Delta^*_{\omega_0}$ satisfying~\eqref{BR2}, and fix $Q_0$ sufficiently large so that~\eqref{eqQ} is satisfied. Assume that
\begin{equation}\label{cond2}
\varepsilon r^{-1} \PM h \PM \Delta(Q_0)^{-1}.
\end{equation}
Then, for each $i\in \N$, there exists a real-analytic transformation 
\[ \mathcal{F}^i : D_{r_i,s_i} \times O_{h_i} \rightarrow D_{r, s} \times O_{h}, \]
such that $H \circ \mathcal{F}^i=N^i + P^i$ with $N^i(I,\omega)=e^i(\omega)+\omega\cdot I$ and $|P^i|_{r_i, s_i,h_i} \leq \varepsilon_i$. Moreover, we have the estimates
\[ |\bar{W}_0(\mathcal{F}^{i+1}-\mathcal{F}^{i})|_{r_{i+1}, s_{i+1},h_{i+1}} \MP \varepsilon_i(r_i h_i)^{-1} \]
where $\bar{W}_0=\mathrm{Diag}(r_0^{-1}\mathrm{Id},\sigma_0^{-1}\mathrm{Id},h_0^{-1}\mathrm{Id})$.
\end{proposition}

\begin{proof}
For $i=0$, we let $\mathcal{F}^0$ be the identity and there is nothing to prove. The general case follows by an easy induction. Indeed, assume that the statement holds true for some $i\in \N$, and let $H_i=H \circ \mathcal{F}^i=N^i + P^i$, defined on $D_{r_i,s_i} \times O_{h_i}$. We want to apply Proposition~\eqref{kamstep} to this Hamiltonian, with $\varepsilon=\varepsilon_i$, $r=r_i$, $s=s_i$, $h=h_i$, $\sigma=\sigma_{i}$ and $Q=Q_i$. First, we need to check that $0<\sigma_{i}< s_i$ and $1 \PM Q_i$. The first condition is equivalent to $\sum_{l=0}^{i}\sigma_l < s$, whereas the second condition is implied by $1 \PM Q_0$, and it is easy to see that both conditions are satisfied by the choice of $Q_0$ in~\eqref{eqQ}, as $s<1$. Then we need to check that the conditions  
\begin{equation*}
\varepsilon_i r_i^{-1} \PM h_i \PM \Delta(Q_i)^{-1}, \quad 1 \PM Q_i\sigma_i,
\end{equation*}
are satisfied. Since 
\begin{equation}\label{eqdelta}
\Delta(Q_i) = \Delta(\Delta^*(\Delta_i)) \leq \Delta_i, 
\end{equation}
it is sufficient to check the conditions
\begin{equation}\label{condit}
\varepsilon_i r_i^{-1} \PM h_i \PM \Delta_i^{-1}, \quad 1 \PM Q_i\sigma_i.
\end{equation}
The second condition of~\eqref{condit} is satisfied, for all $i \in \N$, simply by the choice of the constant in the definition of $\sigma_i$. As for the first condition of~\eqref{condit}, it is satisfied for $i=0$ by~\eqref{cond2}, and since the sequences $\varepsilon_i r_i^{-1}$, $h_i$ and $\Delta_i^{-1}$ decrease at a geometric rate with respective ratio $1/8$, $1/4$ and $1/2$, the first condition of~\eqref{condit} is therefore satisfied for any $i \in \N$. Hence Proposition~\eqref{kamstep} can be applied: there exists a real-analytic transformation 
\[ \mathcal{F}_i=(\Phi_i,\varphi_i) : D_{r_{i+1}, s_{i+1}} \times O_{h_{i+1}} \rightarrow D_{r_i, s_i} \times O_{h_i}, \]
such that $H_i \circ \mathcal{F}_i=N^{i,+} + P^{i,+}$ with $N^{i,+}(I,\omega)=e^{i,+}(\omega)+\omega\cdot I$ and $|P^{i,+}|_{r_{i+1}, s_{i+1},h_{i+1}} \leq \varepsilon_{i+1}$. Moreover, we have the estimates
\[ |W_i(\Phi_i-\mathrm{Id})|_{r_{i+1}, s_{i+1},h_{i+1}} \MP (r_i\sigma_i)^{-1}\Psi(Q_i)\varepsilon_i, \quad |W_i(D\Phi-\mathrm{Id})W_i^{-1}|_{r_{i+1}, s_{i+1},h_{i+1}} \MP (r_i\sigma_i)^{-1}\Psi(Q_i)\varepsilon_i, \]
\[ |\varphi_i-\mathrm{Id}|_{h_{i+1}} \MP \varepsilon_i r_i^{-1}, \quad   h_i|D\varphi_i-\mathrm{Id}|_{h_{i+1}} \MP \varepsilon_i r_i^{-1},\]
where $W_i=\mathrm{Diag}(r_i^{-1}\mathrm{Id},\sigma_i^{-1}\mathrm{Id})$. 

We just need to set $\mathcal{F}^{i+1}=\mathcal{F}^{i} \circ \mathcal{F}_{i}$, $N^{i+1}=N^{i,+}$ and $P^{i+1}=P^{i,+}$ to have 
\[ H \circ \mathcal{F}^{i+1}= H_i \circ \mathcal{F}_i=N^{i+1} + P^{i+1} \]
with $N^{i+1}(I,\omega)=e^{i+1}(\omega)+\omega\cdot I$ and $|P^{i+1}|_{r_{i+1}, s_{i+1},h_{i+1}} \leq \varepsilon_{i+1}$.

It remains to estimate $\mathcal{F}^{i+1}-\mathcal{F}^{i}$. Setting $\bar{W}_i=\mathrm{Diag}(r_i^{-1}\mathrm{Id},\sigma_i^{-1}\mathrm{Id},h_i^{-1}\mathrm{Id})$, the estimates above implies
\begin{equation}\label{bla1}
|\bar{W}_i(\mathcal{F}_i - \mathrm{Id})|_{r_{i+1}, s_{i+1},h_{i+1}} \MP \max\{(r_i\sigma_i)^{-1}\Psi(Q_i)\varepsilon_i,\varepsilon_i(r_i h_i)^{-1}\} \MP \varepsilon_i(r_i h_i)^{-1}
\end{equation}
since $\sigma_i^{-1} \Psi(Q_i) = Q_i \Psi(Q_i) \PM h_i^{-1}$, and similarly
\begin{equation}\label{bla2}
|\bar{W}_i(D\mathcal{F}_i - \mathrm{Id})\bar{W}_i^{-1}|_{r_{i+1}, s_{i+1},h_{i+1}} \MP \varepsilon_i(r_i h_i)^{-1}. 
\end{equation}
Using a classical telescoping argument, the fact that $|\bar{W}_{i-1}\bar{W}_i| \leq 1$ and the estimate~\eqref{bla2}, we get
\begin{equation}\label{bounded}
|\bar{W}_0 D\mathcal{F}^i \bar{W}_i^{-1}|_{r_{i+1}, s_{i+1},h_{i+1}} \MP \prod_{l=0}^{i} \left(1+\varepsilon_l(r_l h_l)^{-1}\right) \MP 1
\end{equation}
as $\varepsilon_l(r_l h_l)^{-1}$ decreases geometrically. Then, from the mean value theorem, we have
\begin{eqnarray*}
|\bar{W}_0(\mathcal{F}^{i+1}-\mathcal{F}^{i})|_{r_{i+1}, s_{i+1},h_{i+1}} & = & |\bar{W}_0(\mathcal{F}^{i} \circ \mathcal{F}_i-\mathcal{F}^{i})|_{r_{i+1}, s_{i+1},h_{i+1}} \\ 
& \leq & |\bar{W}_0 D\mathcal{F}^i \bar{W}_i^{-1}|_{r_{i+1}, s_{i+1},h_{i+1}}|\bar{W}_i(\mathcal{F}_i - \mathrm{Id})|_{r_{i+1}, s_{i+1},h_{i+1}},
\end{eqnarray*}
and this estimate, together with~\eqref{bla1} and~\eqref{bounded}, implies
\[ |\bar{W}_0(\mathcal{F}^{i+1}-\mathcal{F}^{i})|_{r_{i+1}, s_{i+1},h_{i+1}} \MP \varepsilon_i(r_i h_i)^{-1} \]
which is the required estimate.
\end{proof}

We can finally conclude the proof of Theorem~\ref{thm2}.

\begin{proof}[Proof of Theorem~\ref{thm2}]
Recall that we are given $\varepsilon,r,s,h$ and that we fixed $\eta$ small enough to define the sequences $\varepsilon_i,r_i,h_i$, and then we chose $Q_0 \geq 1$ satisfying~\eqref{eqQ} to define the sequences $\Delta_i,Q_i,\sigma_i$ and $s_i$. Moreover, we have 
\begin{equation}\label{limits}
\lim_{i \rightarrow +\infty}\varepsilon_i=\lim_{i \rightarrow +\infty}r_i=\lim_{i \rightarrow +\infty}h_i=0, \quad \lim_{i \rightarrow +\infty}s_i\geq s/2. 
\end{equation}
Now the condition~\eqref{cond0} implies the condition~\eqref{cond2} and Proposition~\ref{kamiter} can be applied: for each $i\in \N$, there exists a real-analytic transformation 
\[ \mathcal{F}^i : D_{r_i,s_i} \times O_{h_i} \rightarrow D_{r, s} \times O_{h}, \]
such that $H \circ \mathcal{F}^i=N^i + P^i$ with $N^i(I,\omega)=e^i(\omega)+\omega\cdot I$ and $|P_i|_{r_i, s_i,h_i} \leq \varepsilon_i$. Moreover, we have the estimates
\begin{equation}\label{teles}
|\bar{W}_0(\mathcal{F}^{i+1}-\mathcal{F}^{i})|_{r_{i+1}, s_{i+1},h_{i+1}} \MP \varepsilon_i(r_i h_i)^{-1}. 
\end{equation}
As $\varepsilon_i(r_i h_i)^{-1}$ decreases geometrically, these estimates and~\eqref{limits} show that the transformations $\mathcal{F}^i=(\Phi^i,\varphi^i)$ converge uniformly, as $i$ goes to infinity, to a map 
\[ \mathcal{F}=(\Phi_{\omega_0},\varphi) : \{0\} \times \T^n_{s/2} \times \{\omega_0\}= D_{0,s/2} \times O_0 \rightarrow D_{r,s} \times O_h  \]
which consists of a real map $\varphi : \{\omega_0\}=O_0 \rightarrow O_h$ and a real-analytic embedding
\[ \Phi_{\omega_0} : \T^n_{s/2} \rightarrow D_{r,s}\]
where, for simplicity, we identified $D_{0,s/2}=\{0\} \times \T^n_{s/2}$ with $\T^n_{s/2}$. Note that by reality, $\varphi(\omega_0)=\tilde{\omega} \in \R^n$ and $\Phi_{\omega_0}(\T^n) \subseteq B \times \T^n$. Moreover, from the estimate~\eqref{teles} and a usual telescopic argument,
\[ |\bar{W}_0 (\mathcal{F}-\mathrm{Id})|_{s/2} \MP \varepsilon(r h)^{-1} \]
from which one deduces that
\[ |W (\Phi_{\omega_0}-\Phi_0)|_{s/2} \MP \varepsilon(r h)^{-1}, \quad |\varphi(\omega_0)-\omega_0| \MP \varepsilon r^{-1}, \]
where $W=\mathrm{Diag}(r^{-1}\mathrm{Id},Q_0^{-1}\mathrm{Id})$, since $r_0=r$ and $\sigma_0 = C Q_0^{-1}$.

To conclude, fix $i \in \N$ and $\omega \in O_{h_i}$. Then, denoting $X_{H_\omega}$, $X_{N^i_\omega}$ and $X_{P^i_\omega}$ the Hamiltonian vector fields associated to $H_\omega$, $N^i_\omega$ and $P^i_\omega$, we have
\begin{eqnarray}\label{estfin} \nonumber
X_{H_{\varphi_i(\omega)}} \circ \Phi^i_\omega - D\Phi^i_\omega X_{N^i_\omega} & = & D\Phi^i_\omega \left( (\Phi^i_\omega)^* X_{H_{\varphi_i(\omega)}}-X_{N^i_\omega} \right) \\ \nonumber
& = & D\Phi^i_\omega \left( X_{H_{\varphi_i(\omega)} \circ \Phi^i_\omega}-X_{N^i_\omega} \right) \\
& = & D\Phi^i_\omega X_{P^i_\omega}
\end{eqnarray}
where we used the fact that $\Phi^i_\omega$ is symplectic in the second equality, and the relation 
\[ H_{\varphi_i(\omega)} \circ \Phi^i_\omega - N^i_\omega = P^i_\omega \] 
in the last equality. Using the inequality $|P_i|_{r_i,s_i,h_i} \leq \varepsilon_i$ together with Cauchy's estimate, one obtains that, as $i$ goes to infinity, $X_{P^i_\omega}$ converges to zero uniformly on $D_{0,s/2}$, whereas $D\Phi^i_\omega$ is uniformly bounded by the estimate~\eqref{bounded}. Hence the right-hand side of~\eqref{estfin} converges uniformly to zero, and so does the left-hand side: at the limit we obtain the equality
\[ X_{H_{\tilde{\omega}}} \circ \Phi_{\omega_0} =D\Phi_{\omega_0}.X_{\omega_0} \]
since $X_{N^i_{\omega}}$ converges to the constant vector field $X_{\omega_0}=\omega_0$ on $\T^n$. From this equality it follows that the embedded torus $\Phi_{\omega_0}(\T^n)$ is invariant by the Hamiltonian flow of $H_{\tilde{\omega}}$ and quasi-periodic with frequency $\omega_0$, and this ends the proof.
\end{proof}

\appendix

\section{Technical lemmas}

In this appendix, we state two technical lemmas that were used in the proof of Propostion~\ref{kamstep}. The first one deals with the estimate of the remainder of the Taylor's expansion at order one of an analytic function.

\begin{lemma}\label{tech1}
Let $P$ be an analytic function defined on $D_{r,s}$ and
\[ \bar{P}(I,\theta)=P(0,\theta)+\partial_I P(0,\theta)\cdot I. \]
Then, for any $0<c<1$, we have the estimate
\[ |P-\bar{P}|_{cr,s} \leq c^2(1-c)^{-1}|P|_{r,s}. \]
\end{lemma}

For a proof (of a more general statement), we refer to \cite{Alb07}, Lemma A.5.

Then we need a  quantitative version of the implicit function theorem for a real-analytic map.

\begin{lemma}\label{tech2}
Let $f : O_h \rightarrow \C^n$ be a real-analytic map satisfying $|f-\mathrm{Id}|_h \leq \delta \leq h/4$. Then $f$ has a real-analytic inverse $\varphi : O_{h/4} \rightarrow O_h$ which satisfies
\[ |\varphi-\mathrm{Id}|_{h/4} \leq \delta, \quad h/4|D\varphi-\mathrm{Id}|_{h/4} \leq \delta.  \]
\end{lemma}

For a proof, we refer to \cite{Pos01}, Lemma A.3.

\addcontentsline{toc}{section}{References}
\bibliographystyle{amsalpha}
\bibliography{InvariantTori}

\end{document}